\documentclass[12pt]{article}

\usepackage{amsmath,amsthm,amssymb}
\usepackage{enumerate}
\usepackage{subfigure}
\usepackage{hyperref}
\usepackage[usenames]{color}
\usepackage{pdfpages}
\usepackage{tikz}
\usetikzlibrary{decorations.pathreplacing}

\newcommand{\ac}{{\mathcal{A}}}
\newcommand{\Pp}{{\mathcal{P}}}

\newcommand{\Rck}{{\mathcal{R}}^{(k)}}

\newcommand{\Rr}{{\mathcal{R}}}
\newcommand{\Var} {\mathrm {Var}}
\newcommand{\Prob} {\mathbb {P}}
\newcommand{\prob}[1]{\Prob\left(#1\right)}
\newcommand{\pc}[2]{\prob{#1\,|\,#2}}
\newcommand{\Ec}{\mathcal{E}}
\newcommand{\eps} {\varepsilon}

\newcommand{\E}{{\mathbb{E}}}
\newcommand{\G} {\mathbb G}

\newcommand{\M} {\mathbb M}
\newcommand{\R} {\mathbb R}

\newcommand{\Bi}{\operatorname{Bin}}

\newcommand{\x}{\mathbf{x}}

\newcommand{\y}{\mathbf{y}}

\newcommand{\Gk}{\G^{(k)}}
\newcommand{\Rk}{\R^{(k)}}
\newcommand{\Hn}[1]{\Gk(n,#1)}

\newcommand{\Hnm}{\Hn{m}}
\newcommand{\Hnp}{\Hn{p}}
\newcommand{\hcnd}{\Rk(n,d)}

\newcommand{\cod}{\operatorname{cod}}

\newtheorem{theorem}{Theorem}
\newtheorem{lemma}[theorem]{Lemma}
\newtheorem{corollary}[theorem]{Corollary}
\newtheorem{claim}[theorem]{Claim}
\newtheorem{conjecture}{Conjecture}
\newtheorem*{remark*}{Remark}

\begin{document}
\title{Embedding the Erd\H{o}s-R\'enyi Hypergraph into the Random Regular Hypergraph and Hamiltonicity}
\date{September 23, 2016}

\author{
\normalsize ANDRZEJ DUDEK{$^1$}\thanks{Project sponsored by the National Security Agency under Grant Number H98230-15-1-0172. The United States Government is authorized to reproduce and distribute reprints notwithstanding any copyright notation hereon.} \footnotemark[5] 
\and \normalsize ALAN FRIEZE{$^{2}$}\thanks{Supported in part by NSF grant {CCF 1013110}.} 
\and \normalsize ANDRZEJ RUCI\'NSKI{$^{3}$}\thanks{Supported in part by the Polish NSC grant 2014/15/B/ST1/01688 and NSF grant DMS 1102086.} \footnotemark[5] \and 
\normalsize MATAS \v{S}ILEIKIS{$^{4}$}\thanks{Part of research performed at Uppsala University (Sweden) and the University of Oxford (United Kingdom).} \thanks{Part of research performed during a visit to the Institut Mittag-Leffler (Djursholm, Sweden).} \\
\small {$^1$}Department of Mathematics, Western Michigan University, Kalamazoo, MI \\
\small \texttt{andrzej.dudek@wmich.edu} \\
\small {$^2$}Department of Mathematical Sciences, Carnegie Mellon University, Pittsburgh, PA\\
\small \texttt{alan@random.math.cmu.edu} \\
\small {$^3$}Department of Discrete Mathematics, Adam Mickiewicz University, Pozna\'n, Poland\\
\small \texttt{rucinski@amu.edu.pl} \\
\small {$^4$}Department of Applied Mathematics, Charles University, Prague, Czech Republic\\
\small \texttt{matas.sileikis@gmail.com}
}

\maketitle

\begin{abstract}
We establish an inclusion relation between two uniform models of random $k$-graphs (for constant $k
\ge 2$) on $n$ labeled vertices: $\Gk(n,m)$, the random $k$-graph with $m$ edges, and $\Rk(n,d)$,
the random $d$-regular $k$-graph. We show that if $n\log n\ll m\ll n^k$ we can choose $d = d(n) \sim {km}/n$ and couple
$\Gk(n,m)$ and $\Rk(n,d)$  so that the latter contains the former  with
probability tending to one as $n\to\infty$. This extends an earlier result of Kim and Vu
about ``sandwiching random graphs''. In view of known threshold theorems on the existence of
different types of Hamilton cycles in $\Gk(n,m)$, our result allows us to find conditions under
which $\Rk(n,d)$ is Hamiltonian. In particular, for $k\ge 3$ we conclude that  if $n^{k-2} \ll d
\ll n^{k-1}$, then a.a.s. $\Rk(n,d)$ contains a tight Hamilton cycle.

\end{abstract}

\section{Introduction}
\subsection{Background}
A \emph{$k$-uniform hypergraph} (or \emph{$k$-graph} for short)  on a vertex set $V = [n] = \{1, \dots,
n\}$ is an ordered pair $G = (V,E)$ where $E$ is a family of $k$-element subsets of $V$. The \emph{degree} of
a vertex $v$ in $G$ is defined as
  \begin{equation*}
    \deg_G(v) := |\left\{ e \in E : v \in e \right\}|.
  \end{equation*}
A $k$-graph  is \emph{$d$-regular} if the degree of every vertex is $d$.
  Let $\Rck(n,d)$ be the family of all $d$-regular $k$-graphs on $V$.
  Throughout, we tacitly assume that $k$ divides $nd$.
  By $\Rk(n,d)$ we denote the \emph{$d$-regular random $k$-graph} which
 is chosen uniformly at random from $\Rck(n,d)$.

Let us recall two more standard models of random $k$-graphs on $n$ vertices. For $p \in [0,1]$, the
\emph{binomial random $k$-graph} $\Gk(n,p)$ is obtained by including each of the
$\binom n k$ possible edges with probability $p$, independently of others. Further, for an integer
$m \in [0,\binom n k]$, the \emph{uniform random $k$-graph} $\Gk(n,m)$ is chosen uniformly at random
among all $\binom{\binom nk}m$ $k$-graphs on $V$ with precisely $m$ edges.

We study the behavior of these random $k$-graphs as $n \to \infty$. Parameters $d, m, p$ are
treated as functions of $n$ and typically tend to infinity in case of $d$, $m$, or zero, in case of $p$. Given a
sequence of events $(\ac_n)$, we say that $\ac_n$ holds \emph{asymptotically almost surely}
(\emph{a.a.s.}) if $\prob{\ac_n} \to 1$, as $n \to \infty$. Also, we write $a_n \ll b_n$ and $b_n \gg a_n$ for $a_n = o(b_n)$.

In 2004, Kim and Vu \cite{KV04} proved that if $\log n \ll d \ll n^{1/3}/\log^2 n$ then there exists a coupling (that is, a joint distribution) of the random graphs
$\G^{(2)}(n,p)$ and $\R^{(2)}(n,d)$ with $p=\frac{d}{n}\left(1-O\left((\log n/d)^{1/3}\right)\right)$ such that
\begin{equation}\label{eq:graphembed}
\G^{(2)}(n,p)\subset \R^{(2)}(n,d) \quad \text{a.a.s}.
\end{equation}
They pointed out several consequences of this result, emphasizing the ease with which one can carry
over known  properties of $\G^{(2)}(n,p)$ to the harder to study regular model $\R^{(2)}(n,d)$. Kim
and Vu  conjectured that such a coupling is possible for all $d\gg\log n$ (they also conjectured a
reverse embedding which is not of our interest here). In \cite{DFRS} we considered a
(slightly weaker) extension of Kim and Vu's result to $k$-graphs, $k \ge 3$, and proved that
\begin{equation}\label{eq:hypembed}
  \G^{(k)}(n,m) \subset \R^{(k)}(n,d) \quad\text{a.a.s.}
\end{equation}
whenever $C \log n \le d \ll n^{1/2}$ and $m \sim cnd$ for some absolute large constant $C$ and a
sufficiently small constant $c = c(k)
> 0$.
Although \eqref{eq:hypembed} is stated for the uniform $k$-graph $\Gk(n,m)$, it is easy to see that
one can replace $\Gk(n,m)$ by $\Gk(n,p)$ with $p = m/\binom n k$ (see 
Section \ref{sec:concluding}).

\subsection{The Main Result}
In this paper we extend \eqref{eq:hypembed} to larger degrees, assuming only $d \le c n^{k-1}$ for
some constant $c = c(k)$. Moreover, our result implies that, provided $\log n \ll d \ll n^{k-1}$,
we can take $m \sim nd/k$, that is, the embedded $k$-graph contains almost all edges of the regular
$k$-graph rather than just a positive fraction, as in \cite{DFRS}. The new result is also
valid for $k = 2$ (for the proof of this case alone, see also \cite[Section 10.3]{FK}), and thus extends~\eqref{eq:graphembed}.

\begin{theorem}\label{thm:embed}
    For each $k \ge 2$ there is a positive constant $C$ such that if for some real
    $\gamma=\gamma(n)$ and positive integer $d=d(n)$,
\begin{equation}\label{eq:gamma}
C \left( \left(d/n^{k-1} + (\log n)/d\right)^{1/3} +1/n\right) \le \gamma < 1,
\end{equation}
and $m = (1-\gamma) nd/k$ is an integer, then there is a joint distribution of $\Gk(n,m)$ and
$\Rk(n,d)$ with
    \begin{equation*}
        \lim_{n\to\infty}\prob{\Gk(n,m) \subset \Rk(n,d)} = 1.
    \end{equation*}
\end{theorem}

\begin{remark*}
In the assumption (\ref{eq:gamma}) of Theorem~\ref{thm:embed} the term $1/n$ can be omited when $k\le 7$. Indeed,
the inequality of arithmetic and geometric means implies that
$$(d/n^{k-1} + (\log n)/d)^{1/3} \ge
(2/n^{(k-1)/2})^{1/3}\ge 1/n.$$
\end{remark*}

 For a given $k\ge2$, \emph{a $k$-graph property} is a family of $k$-graphs closed under isomorphisms. A
  $k$-graph property
  $\Pp$ is called \emph{monotone increasing}
  if it is preserved by adding edges (but not necessarily by adding vertices, as the example of, say, perfect matching shows).
\begin{corollary}\label{cor:properties}
Let $\Pp$ be a monotone increasing property of $k$-graphs and $\log n\ll d\ll n^{k-1}$. If
for some $m\le (1-\gamma) nd/k$, where $\gamma$ satisfies (\ref{eq:gamma}),  $\Gk(n,m) \in \Pp$ a.a.s., then $\Rk(n,d) \in \Pp$ a.a.s.
\end{corollary}

\subsection{Comparison with the Proof by Kim and Vu}
Kim and Vu \cite{KV04} proved \eqref{eq:graphembed} by analysing a certain algorithm that generates a random graph $\R_A$, coupling it with $\R(n,d)$ so that $\R_A = \R(n,d)$ a.a.s., and then embedding $\G(n,p)$ into $\R_A$ a.s.s.

 The algorithm can be described concisely as sequentially generated configuration model which rejects a chosen edge (with replacement), if it violates the simplicity of the graph. Note that the algorithm may run out of admissible edges before it produces a $d$-regular graph. Refining analysis of Steger and Wormald \cite{SW}, Kim and Vu \cite{KV04} proved a coupling of $\R_A$ and $\R(n,d)$ for $d \ll n^{1/3}/\log^2 n$. It is worth mentioning that in another paper Kim and Vu \cite{KV06} proved, for $d = n^{1/3 - \eps}$ with arbitrary $\eps > 0$, a slightly stronger statement that $\R_A$ is asymptotically uniform, that is, 
\begin{equation}\label{eq:assuni}
  \prob{\R_A = G} = (1+o(1))|\Rr(n,d)|^{-1}
\end{equation}
uniformly over \emph{all} $G \in \Rr(n,d)$. The last section in \cite{KV06} reflects some beliefs that this result cannot be extended to $d$ larger than $n^{1/3}$. Although for the coupling of $\R_A$ and $\R(n,d)$ it is enough to prove weaker uniformity, when \eqref{eq:assuni} is allowed to fail for $o(|\Rr(n,d)|)$ graphs, an attempt to extend the approach of Kim and Vu did not seem to be very promising.

Another looming obstacle was the dependence of the proof of asymptotic uniformity in \cite{KV04} on an asymptotic formula for $|\Rr(n,d)|$ due to McKay and Wormald \cite{MW91}, which is valid for $d \ll n^{1/2}$. The problem of asymptotically enumerating $\Rr(n,d)$ had been open in the range $n^{1/2} \le d \ll n/\log n$ since 1991 (see \cite{MW91}).

In the present paper we avoid both explicit generation of random regular graphs and enumeration of regular graphs. Instead we embed $\G(n,m)$ directly into $\R(n,d)$. For this we show that if $\R(n,d)$ is revealed edge by edge (by first sampling the graph and then exposing its edges in a random order), then the conditional distribution of the next edge is nearly uniform over the complement of the current graph (unless we are close to the end). 

Still, getting a fair estimate for the conditional distribution of the next edge is as hard as enumerating graphs with a given degree sequence. We deal with this issue by instead estimating ratios of (conditional) probabilities. This allows us to replace asymptotic enumeration by \emph{relative} enumeration, by which we mean comparison of the number of ways to extend two graphs $G_1$, $G_2$ (differing just by two edges) to a $d$-regular graph.

In April 2016, well after the present paper was submitted, Wormald \cite{W16} announced a proof (as a joint result with Anita Liebenau) of asymptotic enumeration in the missing range of $d$. This makes it more likely that an approach relying on enumeration could lead to another proof of our result. However, we have not attempted this.

For the outline of our proof, see Subsection \ref{subsec:structure}.

\subsection{Hamilton Cycles in Hypergraphs}
To show a more specific application of Theorem~\ref{thm:embed}
 we consider Hamilton cycles in random regular hypergraphs.

For integers $1\le\ell< k$, define an {\em $\ell$-overlapping cycle} (or \emph{$\ell$-cycle}, for short) as a $k$-graph in
which, for some cyclic ordering of its vertices, every edge consists of $k$ consecutive vertices,
and every two consecutive edges (in the natural ordering of the edges induced by the ordering of
the vertices) share exactly $\ell$ vertices. (For $\ell>k/2$ it implies, of course, that some nonconsecutive edges intersect as well.)
 A $1$-cycle is called  \emph{loose} and a
$(k-1)$-cycle is called \emph{tight}. A spanning $\ell$-cycle in a $k$-graph $H$ is called an
\emph{$\ell$-Hamilton} cycle. Observe that a necessary condition for the existence of an
$\ell$-Hamilton cycle is that $n$ is divisible by $k-\ell$. We will assume this divisibility
condition whenever relevant.

Let us recall the results on Hamiltonicity of random regular graphs, that is, the case $k=2$. Asymptotically almost sure Hamiltonicity of
$\R^{(2)}(n,d)$ was proved by Robinson and Wormald \cite{RW94} for fixed $d \ge 3$, by
Krivelevich, Sudakov, Vu and Wormald \cite{KSVW}  for $d \ge n^{1/2} \log n$, and by Cooper,
Frieze and Reed \cite{CFR}  for $C \le d \le n/C$ and some large constant $C$.

Much less is known for random hypergraphs. For the binomial models, the thresholds were found only recently. First, results on loose Hamiltonicity of $\Gk(n,p)$ were
obtained by Frieze~\cite{F} (for $k=3$), Dudek and Frieze~\cite{DF} (for $k \ge 4$ and $2(k-1)|n$),
and   by Dudek, Frieze, Loh and Speiss~\cite{DFLS} (for $k \ge 3$ and $(k-1)|n$). As usual,  the
asymptotic equivalence of the models $\Hnp$ and $\Hnm$ (see, e.g., Corollary 1.16 in \cite{JLR})
allows us to reformulate the aforementioned results for the random $k$-graph $\Gk(n,m)$.

\begin{theorem}[\cite{F,DF,DFLS}]\label{thm:gnm_loose}
There is a constant $C > 0$ such that if $m \ge Cn\log n$, then a.a.s. $\G^{(3)}(n,m)$ contains a
loose Hamilton cycle. Furthermore, for every $k \ge 4$ if $m\gg n\log n$, then a.a.s. $\Gk(n,m)$
contains a loose Hamilton cycle.
\end{theorem}

\noindent From Theorem \ref{thm:gnm_loose} and the older embedding result \eqref{eq:hypembed}, in
\cite{DFRS} we concluded that there is a constant $C > 0$ such that if $C\log n \le d \ll
n^{1/2}$, then a.a.s. $\G^{(3)}(n,d)$ contains a loose Hamilton cycle. Furthermore, for every $k
\ge 4$ if $\log n \ll d \ll n^{1/2}$, then a.a.s. $\Rk(n,d)$ contains a loose Hamilton cycle.

Thresholds for $\ell$-Hamiltonicity of $\Gk(n,m)$ in the remaining cases, that is, for $\ell\ge 2$,
were recently determined by Dudek and Frieze~\cite{DF2} (see also Allen, B\"ottcher, Kohayakawa,
and Person~\cite{ABKP}).

\begin{theorem}[\cite{DF2}]\label{thm:gnm_tight}
\ \\[-0.3in]
\begin{enumerate}[(i)]

\item If $k>\ell=2$ and $m\gg n^2$,
then a.a.s. $\Gk(n,m)$ is $2$-Hamiltonian.

\item For all integers $k>\ell \ge 3$, there exists a constant $C$ such that if $m\geq Cn^{\ell}$
then a.a.s. $\Gk(n,m)$ is $\ell$-Hamiltonian.

\end{enumerate}
\end{theorem}
\noindent
In view of Corollary~\ref{cor:properties}, Theorems~\ref{thm:gnm_loose} and~\ref{thm:gnm_tight} immediately imply the following result that was already anticipated by the authors in~\cite{DFRS}.

\begin{theorem}\label{thm:hamil}
\ \\[-0.3in]
\begin{enumerate}[(i)]

\item \label{thm:hamil:1} There is a constant $C > 0$ such that if $C\log n \le d \le n^{k-1}/C$, then a.a.s. $\R^{(3)}(n,d)$ contains a loose Hamilton cycle.
Furthermore, for every $k \ge 4$ there is a constant $C > 0$ such that if $\log n \ll d \le n^{k-1}/C$, then a.a.s. $\Rk(n,d)$ contains a loose Hamilton cycle.

\item \label{thm:hamil:2} For all integers $k>\ell=2$ there is a constant $C$ such that if $n\ll d \le n^{k-1}/C$
then a.a.s. $\Rk(n,d)$ contains a $2$-Hamilton cycle.

\item \label{thm:hamil:3} For all integers $k>\ell \ge 3$ there is a constant $C$ such that if $Cn^{\ell-1} \le d \le n^{k-1}/C$
then a.a.s. $\Rk(n,d)$ contains an $\ell$-Hamilton cycle.
\end{enumerate}
\end{theorem}
We conjecture that in the cases (ii) and (iii) (but not (i)) the assumed lower bound for $d$ is actually a threshold for Hamiltonicity in $\hcnd$, see Section \ref{sec:concluding}.
\noindent

\subsection{Structure of the Paper}\label{subsec:structure}
In the following section we define a $k$-graph process $(\R(t))_t$ which reveals edges of the random
$d$-regular $k$-graph one at a time. Then we state a crucial Lemma \ref{lem:Tlives}, which says, loosely
speaking, that unless we are very close to the end of the process, the conditional distribution of
the $(t+1)$-th edge is approximately uniform over the complement of $\R(t)$. Based on  Lemma
\ref{lem:Tlives}, we show that a.a.s. $\Gk(n,m)$ can be embedded in $\Rk(n,d)$,  by refining a
coupling similar to the one the we used in \cite{DFRS} and thus proving Theorem \ref{thm:embed}.

In Section \ref{sec:lemmaprep} we prove auxiliary results needed in the proof
 of Lemma
\ref{lem:Tlives}. They mainly reflect the phenomenon that a typical trajectory of the $d$-regular
process $(\R(t))_t$ has concentrated local parameters. In particular, concentration of vertex
degrees is deduced from a Chernoff-type inequality (the only ``external'' result used in the
paper), while (one-sided) concentration of common degrees of sets of vertices  is obtained by an
application of the switching technique (a similar application appeared in \cite{KSVW}).

In Section \ref{sec:lemmaproof} we prove
 Lemma \ref{lem:Tlives}. First we rephrase it as an
 enumerative problem (counting the number of $d$-regular extensions of a given $k$-graph).
 We prove Lemma \ref{lem:Tlives} by estimating the ratio of the numbers of extensions of two $k$-graphs which differ just in two edges.
   For this we define two random multi-$k$-graphs (via the configuration model) and couple them using yet
   another switching.

\section{Proof of  Theorem \ref{thm:embed}}\label{sec:pfThm}
We often drop the superscript in notations like $\Gk$ and $\Rk$
whenever $k$ is clear from the context.

Let $K_n$  denote the complete $k$-graph on vertex set $[n]$. Recall the standard $k$-graph
process $\G(t), t = 0, \dots, \binom n k$ which starts with the empty $k$-graph $\G(0) =
([n],\emptyset)$ and at each time step $t \ge 1$ adds an edge $\eps_t$ drawn from $K_n
\setminus \G(t-1)$ uniformly at random. We treat $\G(t)$ as an \emph{ordered $k$-graph} (that is, with an ordering of edges)
and write
\begin{equation*}
\G(t) = (\eps_1, \dots, \eps_t), \qquad t=0,\dots, \binom n k.
\end{equation*}
Of course,  the random uniform $k$-graph $\G(n,m)$ can be obtained from $\G(m)$ by
ignoring the ordering of the edges.

Our approach is to represent $\R(n,d)$ as an outcome of another $k$-graph process which, to some
extent, behaves similarly to $(\G(t))_t$. For this,  generate
a random $d$-regular $k$-graph $\R(n,d)$ and choose an ordering $(\eta_1, \dots, \eta_M)$ of its
$$M:=\frac{nd}k$$ edges uniformly at random. Revealing the edges of $\R(n,d)$ in that order one by one,
 we obtain a
\emph{regular $k$-graph} process
\begin{equation*}
\R(t) = (\eta_1, \dots, \eta_t), \qquad t = 0, \dots, M.
\end{equation*}

For every ordered $k$-graph $G$ with $t$ edges and every edge $e \in K_n \setminus G$ we clearly have
\begin{equation*}
\pc{\eps_{t+1} = e}{\G(t) = G} = \frac{1}{\binom n k - t}.
\end{equation*}
This is not true for $\R(t)$, except for the very first step $t=0$. However, it turns out that for
the most of the time, the conditional distribution of the next edge in the process $\R(t)$ is
approximately uniform, which is made precise by the lemma below. To formulate it we need some more
definitions.

Given an ordered $k$-graph $G$ , let $\Rr_G(n,d)$ be the family of \emph{extensions} of $G$, that is,
ordered $d$-regular $k$-graphs the first edges of which are equal to $G$. More precisely, setting
$G=(e_1,\dots,e_t)$,
$$\Rr_G(n,d)=\{H=(f_1,\dots,f_M): f_i=e_i, i=1,\dots,t,\mbox{ and }
H\in\Rck(n,d)\}.$$ We say that a $k$-graph $G$ with $t \le M$ edges is
\emph{admissible}, if $\Rr_G(n,d)\neq\emptyset$ or, equivalently, $\prob{\R(t) = G} > 0$. We define, for admissible $G$,
\begin{equation}\label{eq:pt}
p_{t+1}(e|G) := \pc{\eta_{t+1} = e}{\R(t) = G}, \quad t =0, \dots, M - 1.
\end{equation}
Given $\epsilon \in (0,1)$, we define  events
\begin{equation}\label{eq:almostuniform}
 \ac_t = \Big\{ p_{t+1}(e|\R(t)) \ge \frac{1-\epsilon}{\binom n k - t} \text{ for every } e \in K_n \setminus \R(t)\Big\}, \quad t = 0, \dots, M - 1.
\end{equation}
Now we are ready to state the main ingredient of the proof of Theorem \ref{thm:embed}.
\begin{lemma}
\label{lem:Tlives} 
Suppose that $\epsilon = \epsilon(n) \in (0,1)$ is such that $(1-\epsilon)M$ is an integer, and consider the event
\begin{equation*}
 \ac := \ac_0 \cap \dots \cap \ac_{(1-\epsilon)M - 1}.
\end{equation*}
For every $k \ge 2$ there is a positive constant $C'$ such that whenever
  $\epsilon$ and $d=d(n)$ satisfy
\begin{equation}\label{eq:epsilon}
C' \left( (d/n^{k-1} + (\log n)/d)^{1/3} + 1/n\right) \le \epsilon < 1
\end{equation} then the event $\ac$ occurs \text{a.a.s.}
\end{lemma}
 From Lemma \ref{lem:Tlives}, which is proved in Section \ref{sec:lemmaproof}, we deduce Theorem \ref{thm:embed}
  using a coupling similar to the one which was used in~\cite{DFRS}.
\begin{proof}[Proof of Theorem \ref{thm:embed}]
Clearly, we can pick $\epsilon \le \gamma /3$ such that $(1-\epsilon)M$ is integer and \eqref{eq:gamma} implies \eqref{eq:epsilon} with $C'$ being some constant multiple of $C$.

Let us first outline the proof. We will define a $k$-graph process $\R'(t) := (\eta'_1, \dots,
\eta'_t), t = 0, \dots, M$ such that for every admissible $k$-graph $G$ with $t \le M -1$ edges,
\begin{equation}\label{eq:etaprime}
    \pc{\eta'_{t+1} = e}{\R'(t) = G} = p_{t+1}(e | G).
\end{equation}
In view of \eqref{eq:etaprime}, the distribution of $\R'(M)$ is the same as the one of $\R(M)$ and
thus we can define $\R(n,d)$ as the $k$-graph $\R'(M)$ with order of edges ignored. Then we will show
that a.a.s. $\G(n,m)$ can be  sampled from the subhypergraph $\R'((1-\epsilon)M)$ of $\R'(M)$.

Now come the details. Set $\R'(0)$ to be an empty vector and define $\R'(t)$ inductively (for $t =
1, 2, \dots$) as follows. Suppose that $k$-graphs $R_t =\R'(t)$ and $G_t = \G(t)$ have been exposed.
Draw $\eps_{t + 1}$ uniformly at random from $K_n \setminus G_t$ and, independently, generate a
Bernoulli random variable $\xi_{t+1}$ with the probability of success  $1 - \epsilon$. If event
$\ac_t$ has occured, that is,
\begin{equation}\label{eq:coupling}
p_{t+1}(e | R_t) \ge \frac{1-\epsilon}{\binom n k - t} \quad \text{ for every } \quad e \in K_n
\setminus R_t,
\end{equation}
then  draw a random edge $\zeta_{t+1} \in K_n \setminus R_t$ according to the distribution
\begin{equation*}
\prob{\zeta_{t+1} = e|\R'(t) = R_t} := \frac{p_{t+1}(e|R_t) - (1-\epsilon)/(\binom n k - t) }{\epsilon} \ge 0,
\end{equation*}
where the inequality holds by \eqref{eq:coupling}. Observe also that
\begin{equation*}
\sum_{e\in K_n \setminus R_t}\prob{\zeta_{t+1} = e|\R'(t) = R_t}=1,
\end{equation*}
so $\zeta_{t+1}$ has a well-defined distribution. Finally, fix an arbitrary bijection $f_{R_t, G_t} : R_t \setminus G_t \to G_t \setminus R_t$ between the sets of edges and define
\begin{equation}\label{coupled}
\eta'_{t+1} = \begin{cases}
\eps_{t+1}, &\text{ if } \xi_{t + 1} = 1, \eps_{t+1} \in  K_n \setminus R_t,\\
f_{R_t, G_t}(\eps_{t+1}), &\text{ if } \xi_{t + 1} = 1, \eps_{t+1} \in R_t,\\
\zeta_{t+1}, &\text{ if } \xi_{t + 1} = 0. \\
\end{cases}
\end{equation}
If the event $\ac_t$ fails, then $\eta'_{t+1}$ is sampled
directly (without defining $\zeta_{t+1}$) according to probabilities \eqref{eq:pt}. Such a
definition of $\eta'_{t+1}$ ensures that
\begin{equation}\label{eq:implies}
\ac_t \cap \{\xi_{t+1} = 1\} \quad \implies \quad \eps_{t+1} \in \R'(t+1).
\end{equation}
 Further, define a random subsequence of edges of $\G((1-\epsilon)M)$,
\begin{equation*}
S := \left\{ \eps_i : \xi_i = 1\;, i\le(1-\epsilon)M \right\}.
\end{equation*}
Conditioning on the vector $(\xi_i)$ determines $|S|$. If $|S| \ge m$, we define $\G(n,m)$ to have the edge set consisting of the first $m$ edges of $S$ (note that since the vectors $(\xi_i)$ and $(\eps_i)$ are independent, these $m$ edges are uniformly distributed), and if $|S| < m$, then we define $\G(n,m)$ as a graph with edges $\{\eps_1, \dots, \eps_m\}$.

Let event $\ac$ be as in Lemma \ref{lem:Tlives}. The crucial thing is that by \eqref{eq:implies} we have
\begin{equation*}
\ac \quad \implies \quad S \subset \R'(M).
\end{equation*}
Therefore
\begin{equation*}
\prob{\G(n,m) \subset \R(n,d)} \ge \prob{\{|S| \ge m\}\cap \ac }.
\end{equation*}
Since by Lemma \ref{lem:Tlives} event $\ac$ holds a.a.s., to complete the proof it suffices to show that
$\prob{|S| < m} \to 0$.

To this end, note that $|S|$ is a binomial random variable, namely,
\begin{equation*}
|S| = \sum_{i = 1}^{ (1-\epsilon)M }\xi_i \sim \Bi( (1-\epsilon)M , 1-\epsilon ),
\end{equation*}
with
\begin{equation}\label{eq:EVar}
\E |S| \ge (1 - 2\epsilon) M \quad \text{and} \quad \Var |S| = (1 - \epsilon)^2\epsilon M \le \epsilon M.
\end{equation}

Recall that $\epsilon \le \gamma/3$ and thus $m = (1-\gamma) M \le  (1 - 3\epsilon)M $. By \eqref{eq:EVar}, Chebyshev's inequality, and the inequality $\epsilon\ge
C'\log n/d$, which follows from~\eqref{eq:epsilon}, we get
\begin{equation}\label{eq:Sm}
    \prob{|S| < m} \le \prob{|S| - \E |S| < - \epsilon M } \le \frac{\epsilon M}{(\epsilon M)^2} = \frac{k}{\epsilon nd} \le  \frac{k}{C'n\log n}\to 0.
\end{equation}
\end{proof}

\section{Preparations for the Proof of Lemma \ref{lem:Tlives}}\label{sec:lemmaprep}
Throughout this section we adopt the assumptions of Lemma \ref{lem:Tlives}, that is, $(1-\epsilon)M$ is
an integer and \eqref{eq:epsilon} holds with a sufficiently large $C' = C'(k) \ge 1$. In
particular,
\begin{equation}\label{eq:logd}
\epsilon \ge C' (\log n / d)^\alpha,
\end{equation}
\begin{equation}\label{eq:dn}
\epsilon \ge C' (d / n^{k-1})^\alpha
\end{equation}
for every $\alpha \ge 1/3$, and
\begin{equation}\label{eq:n}
\epsilon \ge C'/n.
\end{equation}
 Given a $k$-graph $G$ with maximum degree at most $d$, let us define the \emph{residual degree} of a vertex $v\in V(G)$ as
\begin{equation*}
r_G(v) = d - \deg_{G}(v).
\end{equation*}
We begin our preparations toward the proof of Lemma \ref{lem:Tlives}  with a fact which allows one
to control the residual degrees of the evolving $k$-graph $\R(t)=\left( \eta_1, \dots, \eta_t \right)$.
For a vertex $v \in [n]$ and $t = 0, \dots, M$, define random variables
\begin{equation*}
X_t(v) = r_{\R(t)}(v) = | \left\{ i \in (t, M] : v \in \eta_i \right\} |.
\end{equation*}
Given an integer $t \in [0, M]$, we will use a shorthand
\[
\tau=1-\frac tM.
\]
We will usually assume $t \le (1-\epsilon)M$, which implies $\tau \ge \epsilon$.

\begin{claim}
\label{clm:degrees}
 For every $k \ge 2$ there is a constant $a = a(k) > 0$ such that a.a.s.
\begin{equation}
\forall t \le (1-\epsilon)M, \quad \forall v \in [n],\quad |X_t(v) - \tau d| \le \sqrt{a \tau
d\log n}\le \tau d/2 - 1. \label{eq:deguppergeneral}
\end{equation}

\end{claim}

\begin{proof}
A crucial observation is that the concentration of the degrees depends solely on the random
ordering of the edges and not on the structure of the $k$-graph $\R(M)$. If we fix a $d$-regular $k$-graph
$H$ and condition $\R(M)$ to be a random permutation of the edges of $H$, then $X_t(v)$ is a
hypergeometric random variable with expectation
\[
\E X_t(v) = \frac{(M-t)d}{M} = \tau d.
\]
 Using Theorem~2.10 in \cite{JLR} together with inequalities (2.5) and (2.6) therein, we get
\begin{equation*}
\prob{|X_t(v) - \tau d| \ge x}  \le 2 \exp \left\{ - \frac{x^2}{2\tau d \left(1 + x/(3\tau d) \right)} \right\}.
\end{equation*}
Let $a = 3(k+2)$ and $x = \sqrt{a \tau d \log n }$.
Condition \eqref{eq:logd} with $\alpha = 1$ and $C'\ge 9a$ implies that
\begin{equation}\label{eq:taud}
\tau d \ge \epsilon d \ge C' \log n.
\end{equation}
Therefore
\begin{equation}\label{eq:xtau}
x / (\tau d)  = \sqrt{a\log n / (\tau d)} \le \sqrt{a \log n / (\epsilon d)} \le \sqrt{a/C'} \le 1/3.
\end{equation}
Hence,
\begin{equation*}
\prob{|X_t(v) - \tau d| \ge \sqrt{a \tau d \log n }} \le 2\exp \left\{ - \frac{a}{3} \log n \right\} = 2n^{-k - 2}.
\end{equation*}
Since we have fewer than $nM \le n^{k+1}$ choices of $t$ and $v$, the first inequality in
\eqref{eq:deguppergeneral} follows by taking the union bound.

The second inequality in \eqref{eq:deguppergeneral} follows from \eqref{eq:xtau}, since
\begin{equation*}
\sqrt{a \tau d \log n } = x \le \tau d / 3 \le \tau d / 2 - 1,
\end{equation*}
where the last inequality holds (for large enough $n$) by \eqref{eq:taud}.
\end{proof}

Recall that $\Rr_G(n,d)$ is the family of extensions of $G$ to a $d$-regular ordered $k$-graph. For a $k$-graph $H \in \Rr_G(n,d)$ define the \emph{common degree} (relative to subhypergraph $G \subseteq H$) of an ordered pair $(u,v)$ of vertices as
\begin{equation*}
\cod_{H|G}(u,v) = \left| \left\{ W \in \textstyle{\binom {[n]} {k-1}} : W \cup u \in H, W \cup v \in H  \setminus G\right\}\right|.
\end{equation*}
Note that $\cod_{H|G}(u,v)$ is not symmetric in $u$ and $v$. Also, define \emph{the degree of a
pair of vertices} $u, v$ as
\begin{equation*}
\deg_{H}(u,v) = \left| \left\{ e \in H : \{u, v\} \subset e\right\}\right|.
\end{equation*}

\begin{claim} 
\label{lem:codeg} Let $G$ be an admissible $k$-graph with $t + 1 \le (1 - \epsilon)M$ edges such that
\begin{equation}\label{eq:degG}
r_{G}(v) \le  2 \tau d
\qquad \forall  v \in [n].
\end{equation}
Suppose that $\R_{G}$ is a $k$-graph chosen uniformly at random from $\Rr_G(n,d)$. There are constants $C_0, C_1$, and $C_2$, depending on $k$ only such that the following holds.
\\
For each $e \in K_n \setminus G$,
\begin{equation}\label{eq:noedge}
\prob{e \in \R_{G}} \le \frac{C_0\tau d}{n^{k-1}}.
\end{equation}
Moreover, if $\ell \ge \ell_1 := C_1\tau d/n$, then for every $u, v \in [n], u \ne v$,
\begin{equation}\label{eq:pairdegs}
\prob{\deg_{\R_{G} \setminus G}(u,v) > \ell} \le 2^{-(\ell-\ell_1)}.
\end{equation}
Also, if $\ell \ge \ell_2 := C_2\tau d^2 / n^{k-1}$, then for every $u, v \in
[n], u \ne v$,
\begin{equation}\label{eq:degs}
\prob{\cod_{\R_{G}|G}(u,v) > \ell} \le 2^{-(\ell-\ell_2)}.
\end{equation}
\end{claim}
\begin{proof}
To prove \eqref{eq:noedge}, fix $e \in K_n \setminus G$ and define families of ordered $k$-graphs
\begin{equation*}
\Rr_{e \in} = \left\{ H \in \Rr_{G}(n,d) : e \in H \right\} \quad \text{and} \quad \Rr_{e \notin} = \left\{ H \in \Rr_{G}(n,d) : e \notin H \right\}.
\end{equation*}

 In order to compare the sizes of $\Rr_{e \in}$ and $\Rr_{e \notin}$, define an auxiliary bipartite graph $B$ between $\Rr_{e \in}$ and $\Rr_{e \notin}$ in
 which $H \in \Rr_{e \in}$ is connected to $H' \in \Rr_{e \notin}$ whenever $H'$ can be obtained
 from $H$ by the following operation (known as \emph{switching} in the literature dating back to McKay~\cite{M85}).
 Let $e = e_1 = \{v_{1,1}\dots v_{1,k}\}$ and pick $k-1$ more edges
 $$e_i =
 \{v_{i,1}\dots v_{i,k}\} \in H \setminus G,\qquad i = 2, \dots, k$$ 
 (with vertices labeled in the increasing order within each edge) so that all $k$
 edges are disjoint. Replace, for each $j = 1, \dots, k$, the edge $e_j$ by
 $$f_j := \{v_{1,j}\dots v_{k,j}\}$$
  to obtain $H'$ (see Figure~\ref{fig:switching1}).

Let $f(H) = \deg_B(H)$ be the  number of $k$-graphs $H' \in \Rr_{e \notin}$ which can be obtained from $H$, and
$b(H') = \deg_B(H')$ be the number  of $k$-graphs $H \in \Rr_{e \in}$ from which $H'$ can be obtained. Thus,
\begin{equation}\label{eq:doublecounting}
|\Rr_{e \in}| \cdot \min_{H \in \Rr_{e \in}} f(H) \le |E(B)| \le |\Rr_{e \notin}| \cdot \max_{H' \in \Rr_{e \notin}} b(H').
\end{equation}

\begin{figure}
\begin{center}
\subfigure[]{
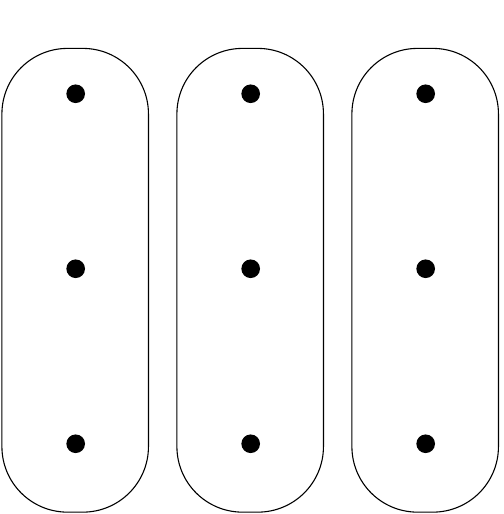
\label{fig:switching1a}
}
\quad\quad\quad
\subfigure[]{
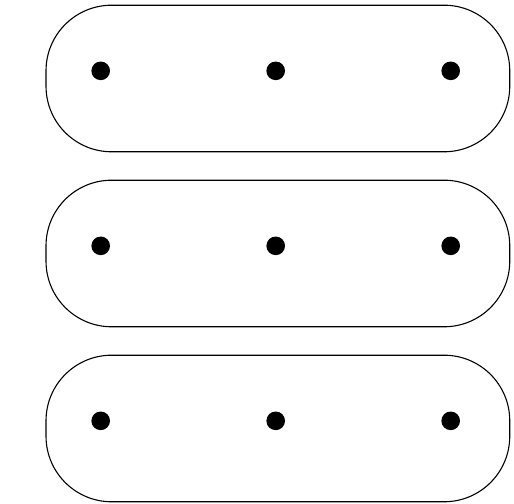
\label{fig:switching1b}
}
\end{center}
\caption{Switching (for $k=3$): before~\subref{fig:switching1a} and after~\subref{fig:switching1b}.}
\label{fig:switching1}
\end{figure}

Note that $H \setminus G$ and $H' \setminus G$ each have $\tau M - 1$ edges and, by
\eqref{eq:degG}, maximum degrees at most $2 \tau d$. To estimate $f(H)$, note that because each
edge intersects at most $k\cdot 2\tau d$ other edges of $H\setminus G$, the number of ways to choose an unordered $(k-1)$-tuple $\{e_2, \dots, e_k\}$ is at least
\begin{equation}\label{eq:kminusone}
\frac{1}{(k-1)!}\prod_{i=1}^{k-1}(\tau M - 1 -i k\cdot2\tau d)\geq (\tau M - k^2\cdot 2\tau d)^{k-1}/(k-1)!.
\end{equation}
We have to subtract the $(k-1)$-tuples which are not allowed since they would create a double edge after the switching (by repeating some edge of $H$ which intersects $e_1$). Their number is at most $kd \cdot (2\tau d)^{k-1}$. Thus,
\begin{align*}
f(H) &\ge \frac{(\tau M - 2k^2\tau d)^{k-1}}{(k-1)!} - k(2\tau)^{k-1}d^k \\
&= \frac{(\tau M)^{k-1}}{(k-1)!} \left( \left(1 - \frac{2k^2d}{M}\right)^{k-1} - \frac{k!(2\tau)^{k-1}d^k}{(\tau M)^{k-1}}\right) \\
&= \frac{(\tau M)^{k-1}}{(k-1)!} \left( \left(1 - \frac{2k^3}{n}\right)^{k-1} - \frac{k! (2k)^{k-1}d}{n^{k-1}}\right) \\
&\ge \frac{(\tau M)^{k-1}}{(k-1)!} \left( 1 - \frac{2k^4}{n} - \frac{(2k)^{2k}d}{n^{k-1}}\right).
\end{align*}
By \eqref{eq:dn} with $\alpha = 1$, \eqref{eq:n}, and sufficiently large $C'$, we have
\begin{equation*}
\frac{2k^4}{n} + \frac{(2k)^{2k}d}{n^{k-1}} \le \frac{\epsilon (2k^4 + (2k)^{2k})}{C'} \le 1/2.
\end{equation*}
Hence,
\begin{equation}\label{eq:fHswitching}
f(H) \ge \frac{(\tau M)^{k-1}}{2(k-1)!}.
\end{equation}
Since $G$ is admissible, either $\Rr_{e \notin}$ or $\Rr_{e \in}$ is non-empty. If $\Rr_{e \in}$ is non-empty, then by \eqref{eq:doublecounting} and the fact that the right-hand side of \eqref{eq:fHswitching} is positive we get that $\Rr_{e \notin}$ is also non-empty.

In order to bound $b(H')$ from above note that there are at most $(2\tau d)^k$ ways to choose a sequence $f_1, \dots, f_k \in H' \setminus G$ such that $v_{1,i} \in f_i$ and we can reconstruct the $k-1$-tuple $e_2, \dots, e_k$ in at most $((k-1)!)^{k-1}$ ways (by fixing an ordering of vertices of $f_1$ and permuting vertices in other $f_i$'s). Therefore $b(H') \le ((k-1)!)^{k-1} \cdot (2\tau d)^k$. This, with \eqref{eq:doublecounting} and \eqref{eq:fHswitching} implies that
\begin{multline*}
\prob{e \in \R_{G}} = \frac{|\Rr_{e \in}|}{|\Rr_G(n,d)|} 
\le \frac{|\Rr_{e \in}|}{|\Rr_{e \notin}|} \\
\le  \frac{\max_{H' \in \Rr_{e \notin}} b(H')}{\min_{H \in \Rr_{e \in}} f(H)}
\le \frac{2((k-1)!)^k(2\tau d)^k}{(\tau M)^{k-1}} = \frac{C_0 \tau d}{n^{k-1}},
\end{multline*}
for some constant $C_0 = C_0(k)$. This concludes the proof of \eqref{eq:noedge}.

\medskip

To prove~\eqref{eq:pairdegs}, fix distinct $u, v \in [n]$, $u < v$, and define the families
\begin{equation*}
\Rr_1(\ell) = \left\{ H \in \Rr_{G}(n,d) : \deg_{H\setminus G}(u,v) = \ell \right\}, \qquad \ell= 0, 1, \dots .
\end{equation*}
Since $G$ is admissible, $\Rr_{G}(n,d)$ is non-empty and thus $\Rr_1(\ell)$ is nonempty for some $\ell \ge 0$. Let $L_1$ be the largest such $\ell$. From the argument below we will see that actually $\Rr_1(\ell)$ is non-empty for every $\ell = 0, \dots, L_1$.

In order to compare sizes of $\Rr_1(\ell)$ and $\Rr_1(\ell-1)$, $\ell \in [1, L]$, we define the following switching which maps a
$k$-graph $H \in \Rr_1(\ell)$ to a $k$-graph $H' \in \Rr_1(\ell-1)$. Select $e_1 \in H \setminus G$ contributing to
$\deg_{H\setminus G}(u,v)$ and pick $k-1$ edges $e_2, \dots, e_k \in H \setminus G$ so that $e_1,
\dots, e_k$ are disjoint. Writing $e_i = \{v_{i,1}\dots v_{i,k}\}, i = 1,\dots,k$ and, for definiteness, labeling vertices inside each $e_i$ in the increasing order, replace $e_1, \dots, e_k$ by $f_1, \dots, f_k$, where $f_j = \{v_{1,j}\dots v_{k,j}\}$, for
$j=1,\dots,k$ (as in  Figure~\ref{fig:switching1}).

Noting that $e_1$ can be chosen in $\ell$ ways, we
get a lower bound on $f(H)$ very similar to that in \eqref{eq:fHswitching}:
\begin{align}
\label{eq:fsecond}
f(H) &\ge \ell\left( (\tau M - 2k^2\tau d)^{k-1}/(k-1)! - k(2\tau)^{k-1}d^k \right) \ge \frac{\ell(\tau M)^{k-1}}{2(k-1)!}.
\end{align}
Since this implies $f(H) > 0$, we get that whenever $\Rr_1(\ell), \ell \ge 1$, then also $\Rr_1(\ell - 1)$ is non-empty. Thus, $\Rr_1(\ell)$ is non-empty for every $\ell = 0, \dots, L_1$, as mentioned above.

For the upper bound for $b(H')$ we choose two disjoint edges in $H'\setminus G$ containing $u$ and~$v$,
respectively, and then $k-2$ more edges in $H'\setminus G$ not containing $u$ and~$v$ so that all edges are disjoint. Crudely bounding number of permutations of vertices inside each of $f_1, \dots, f_k$ by $(k!)^k$, we get $b(H')\le (k!)^k(2\tau d)^2 (\tau M)^{k-2}$. We obtain, for $\ell = \ell_1, \dots, L_1$,
\begin{multline*}
    \frac{|\Rr_1(\ell)|}{|\Rr_1(\ell-1)|} 
		  \le  \frac{\max_{H' \in \Rr_1(\ell-1)} b(H')}{\min_{H \in \Rr_1(\ell)} f(H)} 
		  \le \frac{2(k!)^{k+1}(2\tau d)^2(\tau M)^{k-2}}{\ell(\tau M)^{k-1}} \\
			\le \frac{8(k!)^{k+1}\tau d}{\ell n} 
			\le \frac12,
\end{multline*}
by assumption $\ell \ge \ell_1 = C_1\tau d/n$ and appropriate choice of constant $C_1$. Further,
\begin{multline}\label{eq:prob_cod}
\prob{\deg_{\R_{G} \setminus G}(u,v) > \ell} = \sum_{i = \ell + 1}^{L_1} \frac{|\Rr_1(i)|}{|\Rr_{G}(n,d)|} \le \sum_{i = \ell + 1}^{L_1} \frac{|\Rr_1(i)|}{|\Rr_1(\ell_1)|} \\
= \sum_{i = \ell + 1}^{L_1} \prod_{j = \ell_1 + 1}^i \frac{|\Rr_1(j)|}{|\Rr_1(j-1)|}
\le \sum_{i > \ell} 2^{-(i - \ell_1)} = 2^{-(\ell - \ell_1)},
\end{multline}
which completes the proof of \eqref{eq:pairdegs}.

\medskip

It remains to show \eqref{eq:degs}. Fix an ordered pair $(u,v)$ of distinct vertices and define the families
\begin{equation*}
\Rr_2(\ell) = \left\{ H \in \Rr_{G}(n,d) : \cod_{H|G}(u,v) = \ell \right\}, \qquad \ell= 0, 1, \dots .
\end{equation*}
We compare sizes of $\Rr_2(\ell)$ and $\Rr_2(\ell-1)$ using the following switching.
Select two distinct edges $e_0 \in H$ and $e_1 \in H \setminus G$ contributing to $\cod_{H|G}(u,v)$, that
is, such that $e_0 \setminus \{u\} = e_1 \setminus \{v\}$; pick $k-1$ other edges $e_2, \dots, e_k \in H
\setminus G$ so that $e_1, \dots, e_k$ are disjoint. Writing $e_i = \{v_{i,1}\dots v_{i,k}\},
i = 1,\dots,k$ with $v=v_{1,1}$, replace $e_1, \dots, e_k$ by $f_1, \dots, f_k$, where $f_j = \{v_{1,j}\dots v_{k,j}\}$ for
$j=1,\dots,k$ (see Figure~\ref{fig:switching2}).
We estimate $f(H)$ by first fixing a pair $e_0, e_1$ in one of $\ell$ ways. The number of choices of $e_2, \dots, e_k$ is bounded as in \eqref{eq:kminusone}. However, we subtract not just at most $kd \cdot (2\tau d)^{k-1}$ $(k-1)$-tuples which may create double edges, but also $(k-1)$-tuples for which $(f_1 \setminus \{v\}) \cup \{u\} \in H$ which prevents $\cod(u,v)$ from being decreased. There are at most $d \cdot (2\tau d)^{k-1}$ of such $(k-1)$-tuples. Hence the bound is very similar to \eqref{eq:fsecond} and, omiting very similar calculations, we get
\begin{align*}
f(H) \ge \ell \left( \frac{ (\tau M - k^2 \cdot 2\tau d)^{k-1} }{(k-1)!} - (k+1)d \cdot (2\tau d)^{k-1}\right) \ge \frac{\ell (\tau M)^{k-1}}{2(k-1)!}.
\end{align*} 

Writing $L_2$ for the largest $\ell$ such that $\Rr_2(\ell)$ is non-empty, we get that $\Rr_2(\ell)$ is nonempty for $\ell = 0, \dots, L_2$, by a similar argument as with the previous switching.

\begin{figure}
\begin{center}
\subfigure[]{
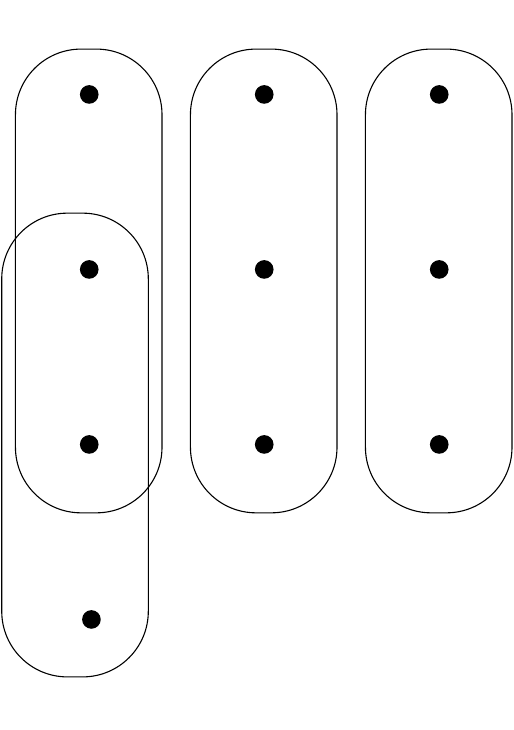
\label{fig:switching2a}
}
\quad\quad\quad
\subfigure[]{
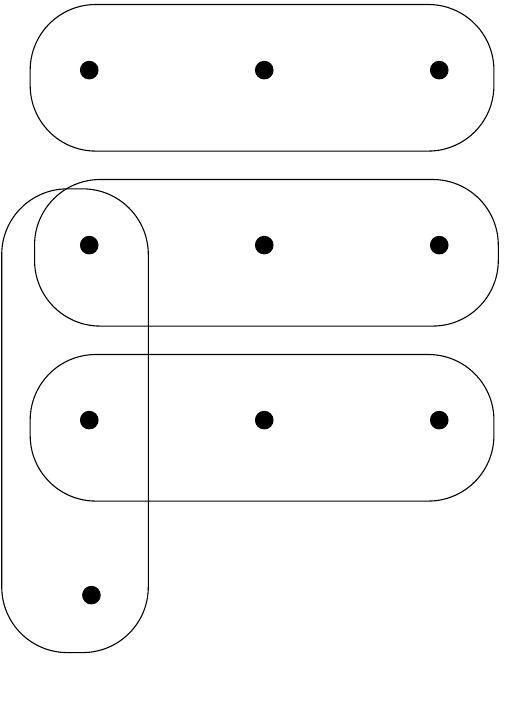
\label{fig:switching2b}
}
\end{center}
\caption{Switching (for $k=3$): before~\subref{fig:switching2a} and after~\subref{fig:switching2b}.}
\label{fig:switching2}
\end{figure}

Conversely, $H$ can be reconstructed from $H'$ by choosing an edge $e_0 \in H'$ containing $u$ but not containing $v$ and then $k$ disjoint edges $f_j \in H' \setminus G$, each containing exactly one vertex from $(e_0 \setminus \{u\}) \cup \{v\}$ and permuting the vertices inside $f_2 \setminus \{v_{1,2}\}, \dots, f_k \setminus \{v_{1,k}\}$ in at most $((k-1)!)^{k-1}$ ways. Therefore $b(H') \le ((k-1)!)^{k-1} d (2\tau d)^k$.
Clearly, for $\ell = \ell_2, \dots, L_2$,
\begin{multline*}
    \frac{|\Rr_2(\ell)|}{|\Rr_2(\ell-1)|} 
		\le \frac{\max_{H' \in \Rr_2(\ell-1)} b(H')}{\min_{H \in \Rr_2(\ell)} f(H)}
		\le\frac{d(2\tau d)^k\cdot2((k-1)!)^k}{\ell(\tau M)^{k-1}} \\
		\le \frac{2^{k+1}((k-1)!)^k k^{k-1}\tau d^2}{n^{k-1}\ell} \le \frac12
\end{multline*}
by the assumption $\ell \ge \ell_2 = C_2\tau d^2 / n^{k-1}$ and appropriate choice of constant $C_2$. Now~\eqref{eq:degs} follows from similar computations to~\eqref{eq:pairdegs}.

 This finishes the proof of Claim~\ref{lem:codeg}.
\end{proof}

\section{ Proof of Lemma \ref{lem:Tlives}}\label{sec:lemmaproof}

In this section we prove the crucial  Lemma~\ref{lem:Tlives}.
In view of Claim \ref{clm:degrees} it suffices to show that
\begin{equation}\label{eq:toshow}
\pc{\eta_{t+1}=e}{\R(t) = G} \ge \frac{1 - \epsilon}{\binom n k - t}, \qquad \forall\;e \in K_n
\setminus G,
\end{equation}
for every $t \le (1-\epsilon)M - 1$ and every admissible $G$ such that
\begin{equation}\label{eq:tautau}
d(\tau - \delta) \le  r_G(v) \le d(\tau + \delta), \qquad v \in [n],
\end{equation}
where
\begin{equation*}
\tau = 1 - t/M \quad \mbox{ and } \quad \delta = \sqrt{a \tau (\log n) /d }.
\end{equation*}
In some cases the following simpler bounds (implied by the second inequality in \eqref{eq:deguppergeneral}) on $r_G(v)$ will suffice:
\begin{equation}\label{eq:degGboth}
\tau d / 2 +1 \le  r_G(v) \le   2\tau d, \qquad v \in [n].
\end{equation}

 Since the average of $\pc{\eta_{t+1} = e}{\R(t) = G}$ over $e \in K_n \setminus G$ is exactly $1/\left(\binom n k - t\right)$, there is $f \in K_n \setminus G$ such that
\begin{equation}\label{eq:avg}
\pc{\eta_{t+1} = f}{\R(t) = G} \ge \frac{1}{\binom n k - t}.
\end{equation}
Fix any such $f$ and let $e \in K_n \setminus G$ be arbitrary. We write $G \cup f$ for an ordered graph obtained by appending edge $f$ at the end of $G$. Setting $\Rr_f :=\Rr_{G \cup f}(n,d)$ and $\Rr_e := \Rr_{G \cup e}(n,d)$, we have
\begin{equation}\label{eq:ratio}
\frac{\pc{\eta_{t+1} = e}{\R(t) = G}}{\pc{\eta_{t+1} = f}{\R(t) = G}} = \frac{| \Rr_{G \cup
e}(n,d)|}{|\Rr_{G \cup f}(n,d)|} = \frac{| \Rr_e|}{| \Rr_f|}.
\end{equation}

To bound this ratio, we need to appeal to the configuration model for hypergraphs. Let $\M_G(n,d)$ be a random \emph{multi-$k$-graph extension} of
$G$ to an ordered $d$-regular multi-$k$-graph. Namely, $\M_G(n,d)$ is a sequence of $M$ edges (each of which is a $k$-element multiset of vertices), the
first $t$ of which comprise $G$, while the remaining ones are generated by taking a random uniform
permutation $\Pi$ of the multiset 
$$\{1,\dots,1,\dots,n,\dots,n\}$$ 
with multiplicities $r_G(v)$, $v
\in [n]$, and splitting it into consecutive $k$-tuples.

The number $N_G$ of such permutations is a multinomial coefficient:
\begin{equation*}
N_G := \binom{k(M-t)}{r_G(1), \dots, r_G(n)} = \frac{\left( k(M - t) \right)!}{ \prod_{v \in [n]} r_G(v)!}.
\end{equation*}
A \emph{loop} is an edge with at least one repeated vertex. We say that an extension is \emph{simple}, if all it's edges are distinct and not loops.

Since each simple extension of $G$ is given by the same number of permutations (namely $(k!)^{M - t}$), $\M_G(n,d)$ is uniform
over $\Rr_G(n,d)$. That is, $\M_G(n,d)$, conditioned on simplicity, has the same distribution as
$\R_G(n,d)$.

Set
\begin{equation*}
\M_e = \M_{G \cup e}(n,d) \quad \text{and} \quad \M_f = \M_{G \cup f}(n,d),
\end{equation*}
 for convenience. Noting that $G \cup f$ has $t + 1$ edges, we have
\begin{equation*}
\prob{\M_f \in \Rr_f} = \frac{|\Rr_f|(k!)^{M-t-1}}{N_{G \cup f}} =  \frac{| \Rr_f|(k!)^{M- t-1}\prod_{v \in [n]}r_{G \cup f}(v)!}{(k(M- t-1))!},
\end{equation*}
and similarly for $\M_e$ and $\Rr_e$. This yields, after a few cancelations, that
\begin{equation}\label{eq:product}
\frac{|\Rr_e|}{|\Rr_f|} = \frac{\prod_{v \in e \setminus f} r_G(v)  }{\prod_{v \in f \setminus
e} r_G(v) } \cdot \frac{\prob{\M_e \in \Rr_e}}{\prob{\M_f \in \Rr_f}} .
\end{equation}
The ratio of the products in \eqref{eq:product} is, by (\ref{eq:tautau}), at least
\begin{equation*}
 \left( \frac{\tau - \delta}{\tau + \delta} \right)^k
 \ge \left( 1 - \frac{2\delta}{\tau} \right)^k
\ge  1 - 2k \sqrt{\frac{a \log n}{\tau d}} \ge  1 - 2k \sqrt{\frac{a \log n}{\epsilon d}} \ge
1-\epsilon/2,
\end{equation*}
where the last inequality holds by \eqref{eq:logd} with $\alpha = 1/3$ and $C'\ge
\sqrt[3]{16ak^2}$. On the other hand, the ratio of probabilities in \eqref{eq:product} will be shown in
Claim \ref{lem:mutualsimplicity} below to be at least $ 1 - \epsilon/2$. Consequently, the entire ratio in \eqref{eq:product}, and thus in  
\eqref{eq:ratio}, will be at least $1 - \epsilon$, which, in view of \eqref{eq:avg}, will imply \eqref{eq:toshow} and yield the lemma.

Hence, to complete the proof of Lemma \ref{lem:Tlives} it remains to  show that the probabilities of
simplicity $\prob{\M_e \in \Rr_e}$ are asymptotically the same for all $e \in K_n \setminus G$. Recall that for
every  edge $e \in K_n \setminus G$ we write
\begin{equation}\label{eq:Gprime}
\M_e = \M_{G \cup e}(n,d) \quad \text{and} \quad 
\Rr_e = \Rr_{G \cup e}(n,d).
\end{equation}

\begin{claim}
\label{lem:mutualsimplicity} If $G$, $e$, and $f$ are as above, then, for every $e \in K_n
\setminus G$,
\begin{equation*}
\frac{ \prob{\M_{e} \in \Rr_{e}} } { \prob{\M_{f} \in \Rr_{f}} } \ge 1 - \epsilon/2.
\end{equation*}
\end{claim}
\begin{proof}
We start by constructing a coupling of $\M_e$ and $\M_f$
 in which they differ in at most $k + 1$ edges (counting in the replacement of $f$ by
  $e$ at the $(t+1)$-th position).

Let $f = \{u_1,\dots, u_k\}$ and $e = \{v_1,\dots, v_k\}$. Further, let $s = k-|f \cap e|$ and suppose
without loss of generality that $\{u_1,\dots, u_s\}\cap\{v_1, \dots, v_s\}=\emptyset$. Let $\Pi_f$
be a random permutation underlying the multi-$k$-graph $\M_f$. Note that $\Pi_f$ differs from any permutation $\Pi_e$ underlying $\M_e$ by having the multiplicities of $v_1,\dots,v_s$ greater by one, and the multiplicities of $u_1,\dots,u_s$ smaller by one than the corresponding multiplicities in $\Pi_e$.

Let $\Pi^*$ be obtained from $\Pi_f$ by  
replacing, for each $i = 1, \dots, s$, a  copy of $v_i$ selected uniformly at random by $u_i$. Define $\M^*$ by
chopping $\Pi^*$ into consecutive $k$-tuples and appending them to $G \cup e$ (see Figure \ref{fig:coupling_seqs}).

\begin{figure}
  \begin{center}
    \begin{tikzpicture}[scale = 1.2,decoration=brace]
      \tikzstyle{vertex}=[draw,circle,fill=white,fill opacity=1,minimum size=4pt, inner sep=3pt]
      \draw (-1,0) node {$\M_f$};
      \foreach \x/\a/\b/\c in {1/*/*/*, 2/*/*/*, 3/*/*/*, 4/*/*/*}
      {
        \draw (\x-0.25,0) node (\x) {$\a$} ;
				\draw (\x,0) node (\x) {$\b$} ;
				\draw (\x+0.25,0) node (\x) {$\c$} ;
      }
			\draw [decorate, very thick, color=black] (0.7,0.3) to node[auto] {$G$} (4.3,0.3);
			
			\draw (5,0) node (5) {$u_1u_2u_3$};
			\draw [decorate, very thick] (4.5,0.3) to node[auto] {$f$} (5.5,0.3);
			
			\foreach \x/\a/\b/\c in {6/*/*/v_2, 7/*/v_1/*, 8/*/*/*, 9/v_3/*/*}
      {
        \draw (\x-0.25,0) node (\x) {$\a$} ;
				\draw (\x,0) node (\x) {$\b$} ;
				\draw (\x+0.25,0) node (\x) {$\c$} ;
			}

			\begin{scope}[shift ={(0,-2)}]
			\draw (-1,0) node {$\M_e$};
      \foreach \x/\a/\b/\c in {1/*/*/*, 2/*/*/*, 3/*/*/*, 4/*/*/*}
      {
        \draw (\x-0.25,0) node (\x) {$\a$} ;
				\draw (\x,0) node (\x) {$\b$} ;
				\draw (\x+0.25,0) node (\x) {$\c$} ;
      }
			\draw [decorate, very thick] (4.3,-0.3) to node[auto] {$G$} (0.7,-0.3);
			
			\draw (5,0) node (5) {$v_1v_2v_3$};
			\draw [decorate, very thick] (5.5,-0.3) to node[auto] {$e$} (4.5,-0.3);
			
			\foreach \x/\a/\b/\c in {6/*/*/u_2, 7/*/u_1/*, 8/*/*/*, 9/u_3/*/*}
      {
        \draw (\x-0.25,0) node (\x) {$\a$} ;
				\draw (\x,0) node (\x) {$\b$} ;
				\draw (\x+0.25,0) node (\x) {$\c$} ;
			}
			\foreach \x in {-0.3,0,0.3}
			{
			  \draw[->, very thick] (5+ \x, 1.7) -- (5 + \x, 0.3);
			}
			\foreach \y in {6.25,7,8.75}
			{
			  \draw[->, very thick] (\y, 1.7) -- (\y, 0.3);
			}
			\end{scope}
      \end{tikzpicture}
    \end{center}
  \caption{Obtaining $\M_e$ from $\M_f$ for $k = s = 3$ by altering the underlying permutation.}
  \label{fig:coupling_seqs}
\end{figure}
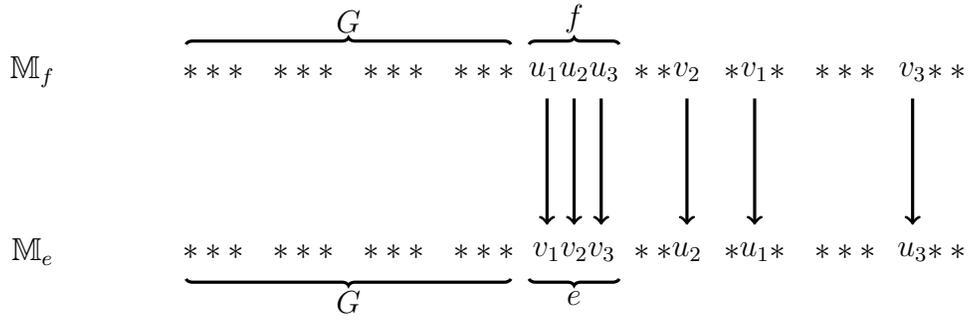

It is easy to see that $\Pi^*$ is uniform over all permutations of the multiset 
$$\left\{ 1,
\dots, 1, \dots, n, \dots, n \right\}$$
 with multiplicities $r_{G \cup e}(v), v \in [n]$. This means that
 $\M^*$ has the same distribution as $\M_e$ and thus we will further identify $\M^*$ and
$\M_e$.

Observe that if we condition $\M_f$ on being a simple $k$-graph $H$, then $\M_e$ can be equivalently obtained by the following switching: (i) replace edge $f$ by $e$; (ii) for each $i = 1, \dots, s$, choose, uniformly at random, an edge $e_i \in H \setminus (G \cup f)$ incident to $v_i$ and replace it by $(e_i \setminus \{v_i\}) \cup \{u_i\}$ (see Figure~\ref{fig:coupling}). Of course, some of $e_i$'s may coincide. For example, if $e_{i_1} = \dots = e_{i_l}$, then the effect of the switching is that $e_{i_1}$ is replaced by
$(e_{i_1} \setminus \{v_{i_1}, \dots, v_{i_l} \}) \cup \{u_{i_1}, \dots, u_{i_l} \}$.

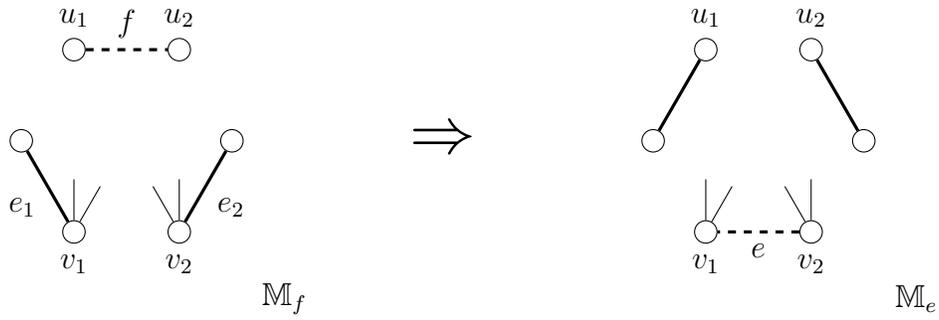
\begin{figure}
\begin{center}
\begin{tikzpicture}[scale = 0.7]
    \tikzstyle{vertex}=[draw,circle,fill=white,fill opacity=1,minimum size=4pt, inner sep=3pt]

      \draw (3,-3) node (a) {$\M_f$};
      \foreach \x in {1,2,4,5}
      {
        \draw (180-\x*60: 2) node[vertex] (\x) {} ;
      }
        \foreach \x/\y in {1/u_1, 2/u_2}
        {
          \draw node [above] at (\x.north) {$\y$};
      }
        \foreach \x/\y in {5/v_1, 4/v_2}
        {
          \draw node [below] at (\x.south) {$\y$};
      }
        \foreach \x/\y in {1/2}
      {
              \draw [dashed, very thick, auto] (\x) to node {$f$} (\y);
      }
        \foreach \x/\y in {4/3, 5/6}
      {
              \draw  (\x) + (60:1) to (\x);
                            \draw  (\x) + (90:1) to (\x);
                            \draw  (\x) + (120:1) to (\x);
      }

        \foreach \x in {3,6}
      {
        \draw (180-\x*60: 2) node[vertex] (\x) {} ;
      }
    \draw [very thick, auto] (3)  to node {$e_2$} (4);
      \draw [very thick, auto] (5) to node {$e_1$} (6);

\draw (6,0) node () {\huge $ \Rightarrow$};

  \begin{scope}[shift={(12,0)}]
    \draw (3,-3) node (a) {$\M_e$};
      \foreach \x in {1,2, ..., 6}
      {
        \draw (180-\x*60: 2) node[vertex] (\x) {} ;
      }
        \foreach \x/\y in {1/u_1, 2/u_2}
        {
          \draw node [above] at (\x.north) {$\y$};
      }
        \foreach \x/\y in {5/v_1, 4/v_2}
        {
          \draw node [below] at (\x.south) {$\y$};
      }
        \foreach \x/\y in {4/5}
      {
              \draw [dashed, very thick, auto] (\x) to node {$e$} (\y);
      }
        \foreach \x/\y in {1/6, 2/3}
      {
              \draw [very thick] (\x) to (\y);
      }
    \draw  (4) + (120:1) to (4);
        \draw  (4) + (90:1) to (4);
        \draw  (5) + (90:1) to (5);
        \draw  (5) + (60:1) to (5);
  \end{scope}
  \end{tikzpicture}
  \end{center}
  \caption{Obtaining $\M_e$ from $\M_f$ for $k = s=2$: only relevant edges are displayed; the ones belonging to $\M_f \setminus (G \cup f)$ are shown as solid lines.}
  \label{fig:coupling}
\end{figure}

  The crucial idea is that such a switching is unlikely to create loops or multiple edges. However,  for certain $H$ this might not true. For example, if $e \in H \setminus (G \cup f)$, then the random choice of $e_i$'s in step (ii) is unlikely to destroy $e$, but in step (i) edge $f$ has been deterministically replaced by an additional copy of $e$, thus creating a double edge. Moreover, if almost every $(k-1)$-tuple of vertices extending $v_i$ to an edge in $H \setminus (G \cup f)$ also extends $u_i$ to an edge in $H$, then most likely the replacement of $v_i$ by $u_i$ will create a double edge, too.
To avoid such and other bad instances, we say that $H \in \Rr_f$ is \emph{nice} if the following three properties hold:	
\begin{equation}\label{eq:nice1}
e \notin H
\end{equation}
\begin{equation}\label{eq:nice3}
\max_{i = 1, \dots, s} \deg_{H \setminus (G \cup f)}(u_i,v_i) \le \ell_1 + k \log_2 n,
\end{equation}
\begin{equation}\label{eq:nice2}
\max_{i = 1, \dots, s} \cod_{H|G \cup f}(u_i,v_i) \le \ell_2 + k \log_2 n,
\end{equation}
where $\ell_1 = C_1\tau d/n$ and $\ell_2 = C_2\tau d^2 / n^{k-1}$ are as in Claim \ref{lem:codeg}.
Note that $\M_f$, conditioned on $\M_f \in \Rr_f$, is distributed uniformly over $\Rr_{G\cup f}(n,d)$. Since we chose $f$ such that by \eqref{eq:avg} is satisfied, we have that $k$-graph $G \cup f$ is admissible. Therefore by Claim \ref{lem:codeg} we have
\begin{align}
\notag \pc{\M_f \text{ is not nice}}{\M_f \in \Rr_f} &\le \frac{C_0 \tau d}{n^{k-1}} + 2 \cdot s 2^{- k
\log_2 n } \\
 \label{eq:codegcond} &\le \frac{C_0d + 2k}{n^{k-1}} \le \frac{\epsilon}{4},
\end{align}
where the last inequality follows by \eqref{eq:dn} with $\alpha = 1$ and sufficiently large constant $C'$. We have
\begin{align}
\notag \frac{ \prob{\M_e \in \Rr_e} } { \prob{\M_f \in \Rr_f} } &\ge \pc{\M_e \in \Rr_e}{\M_f \in \Rr_f} \\
\label{eq:decomposition} &\ge \pc{\M_e \in \Rr_e}{\M_f \text{ is nice}}\pc{\M_f \text{ is nice}}{\M_f \in \Rr_f}.
\end{align}
It suffices to show that
\begin{equation}\label{eq:MMnice}
\pc{\M_e \in \Rr_e}{\M_f \text{ is nice}} \ge 1 - \epsilon/4,
\end{equation}
since in view of \eqref{eq:codegcond} and \eqref{eq:MMnice}, inequality  \eqref{eq:decomposition}
completes the proof of Claim~\ref{lem:mutualsimplicity}.

Now we prove \eqref{eq:MMnice}. Fix a nice $k$-graph $H \in \Rr_f$ and condition on the event $\M_f = H$. The event that
$\M_e$ is not simple is contained in the union of the following four events:
\begin{itemize}
\item[] $\Ec_1 = \{$ two of the randomly chosen edges $e_1, \dots, e_s$ coincide $\}$,
\item[] $\Ec_2 = \{$ $(e_i \setminus v_i) \cup u_i$ is a loop for some $i = 1, \dots, s$ $\}$,
\item[] $\Ec_3 = \{$ $(e_i \setminus v_i) \cup u_i \in H$ for some $i = 1, \dots, s$ $\}$,
\item[] $\Ec_4 = \{$ $(e_i \setminus v_i) \cup u_i = (e_j \setminus v_j) \cup u_j$ for some distinct $i$ and $j$  $\}$.
\end{itemize}
Event $\Ec_1$ covers all cases when a double edge is created by replacing several vertices in the same edge. Creation of multiple edges in other ways is addressed by events $\Ec_3$ and $\Ec_4$.

In what follows we will several times use the fact that
\begin{equation}\label{eq:mindeg}
\deg_{H\setminus (G \cup f)} (v) \ge \tau d/2  \ge \epsilon d/2, \qquad \forall v \in [n],
\end{equation}
which is immediate from \eqref{eq:degGboth} and $\tau\ge \epsilon$. To bound the probability of
$\Ec_1$, observe that, given $1 \le i < j \le s$, the number of choices of a coinciding pair $e_i
= e_j$ is $\deg_{H \setminus (G\cup f)}(v_i,v_j) \le \deg_{H \setminus (G\cup f)}(v_i)$ and the probability that both $v_i$ and $v_j$ actually select a fixed common edge is $(\deg_{H \setminus (G \cup f)}(v_i)\deg_{H \setminus (G \cup f)}(v_j))^{-1}$. Therefore using \eqref{eq:mindeg} we obtain
\begin{multline}\label{eq:Ec1}
\prob{\Ec_1 | \M_f = H} \le \sum_{1 \le i < j \le s} \frac{\deg_{H \setminus (G \cup f)}(v_i,v_j)}{\deg_{H \setminus (G \cup f)}(v_i)\deg_{H \setminus (G \cup f)}(v_j)} \\
\le \sum_{1 \le i < j \le s} \frac{1}{\deg_{H \setminus (G \cup f)}(v_j)} \le \frac{2\binom k 2}{\epsilon d} \le \frac{\epsilon}{16},
\end{multline}
where the last inequality follows from \eqref{eq:logd} with $\alpha = 1/2$ and sufficiently large~$C'$.

To bound the probability of $\Ec_2$, note that a loop in $\M_e$ can only be created  when for
some $i = 1, \dots, s$, the randomly chosen edge $e_i$ contains both $v_i$ and $u_i$. There are
at most $\deg_{H \setminus (G \cup f)}(u_i, v_i)$ such edges. Therefore, by \eqref{eq:nice3} and
\eqref{eq:mindeg} we get
\begin{multline}\label{eq:loops}
\pc{\Ec_2}{\M_f = H} \le \sum_{i = 1}^s \frac{\deg_{H \setminus (G \cup f)}(u_i,v_i)}{\deg_{H\setminus (G \cup f)}(v_i)}
\le \frac{2k(\ell_1 + k \log_2 n)}{\tau d} \\
\le \frac{2k \ell_1}{\tau d} + \frac{2k^2 \log_2 n}{\epsilon d} = \frac{2k C_1}{n} + \frac{2k^2 \log_2
n}{\epsilon d} \le \frac{\epsilon}{16},
\end{multline}
where the last inequality is implied by \eqref{eq:logd} with $\alpha=1/2$, \eqref{eq:n} and sufficiently large~$C'$.

\medskip

Similarly we bound the probability of $\Ec_3$, the event that for some $i$ we will choose $e_i \in H\setminus (G 
\cup f)$ with
 $(e_i \setminus v_i) \cup u_i \in H$. There are $\cod_{H|G\cup
f}(u_i,v_i)$ such edges. Thus, by \eqref{eq:nice2} and  \eqref{eq:mindeg} we obtain
\begin{multline}\label{eq:multiples}
\pc{\Ec_3}{\M_f = H} \le \sum_{i = 1}^s \frac{\cod_{H|G\cup f}(u_i,v_i)}{\deg_{H \setminus (G \cup f)}(v_i)}
\le \frac{2k(\ell_2 + k \log_2 n)}{\tau d} \\
\le \frac{2k\ell_2}{\tau d} + \frac{2k^2 \log_2 n}{\tau d} \le  \frac{2kC_2
d}{n^{k-1}} + \frac{2k^2 \log_2 n}{\epsilon d} \le \frac{\epsilon}{16},
\end{multline}
where the last inequality follows from \eqref{eq:logd} with $\alpha = 1/2$, \eqref{eq:dn} with $\alpha=1$ and sufficiently large $C'$.

Finally, note that, given $1 \le i < j \le s$, if a pair $e_i, e_j \in H \setminus (G
\cup f)$ satisfies the condition in $\Ec_4$, then $u_j \in e_i \setminus v_i$ and $e_j = (e_i \setminus \{v_i, u_j\}) \cup \{v_j, u_i\}$. This means that $e_j$ is uniquely determined by $e_i, u_i, u_j, v_i,$ and $v_j$. Therefore the number of such pairs $e_i, e_j$ is at most $\deg_{H \setminus (G \cup f)}(v_i, u_j) \le \deg_{H \setminus (G \cup f)}(v_i)$ and
we get exactly the same bound as in \eqref{eq:Ec1}:
\begin{multline}\label{eq:twomultiples}
\pc{\Ec_4}{\M_f = H} 
\le \sum_{1 \le i < j \le s} \frac{\deg_{H \setminus (G \cup f)}(v_i, u_j)}{\deg_{H \setminus (G \cup f)}(v_i)\deg_{H \setminus (G \cup f)}(v_j)} \\
\le \sum_{1 \le i < j \le s} \frac{1}{\deg_{H \setminus (G \cup f)}(v_j)}
\le \frac{\epsilon}{16}.
\end{multline}
Combining \eqref{eq:Ec1}-\eqref{eq:twomultiples} and averaging over nice $H$, we obtain \eqref{eq:MMnice}, as required.
\end{proof}

\section{Concluding Remarks}\label{sec:concluding}

Theorem~\ref{thm:embed} remains valid if we replace the random hypergraph $\Gk(n,m)$ by $\Gk(n,p)$ with
$p = (1-2\gamma)d/\binom{n-1}{k-1}$, say. To see this one can modify the proof of Theorem
\ref{thm:embed} as follows. Let $B_n \sim \Bi(\binom n k, p)$ be a random variable independent of
the process $(\G(t))_{t}$. If $B_n \le m \le |S|$, sample $\Gk(n,p)$ by taking the first $B_n$
edges of $S$ (which are uniformly distributed over all $k$-graphs with $B_n$ edges). Otherwise
sample $\Gk(n,p)$ among $k$-graphs with $B_n$ edges independently. In view of the assumption~\eqref{eq:gamma},  Chernoff's inequality (see \cite[(2.5)]{JLR}) and \eqref{eq:Sm} imply
\[
\prob{\Gk(n,p) \not\subset \Rk(n,d)} \le \prob{B_n > m} + \prob{|S| < m} \to 0, \quad \text{as} \quad n \to \infty.
\]

  The lower bound on $d$ in Theorem~\ref{thm:embed}
  is necessary because the second moment method applied to $\Gk(n,p)$ (cf. Theorem 3.1(ii)
   in \cite{B01}) and asymptotic equivalence of $\Gk(n,p)$ and $\Gk(n,m)$ yields that
    for $d = o(\log n)$ and $m \sim cM$ there is a sequence $\Delta = \Delta(n)\gg d$
    such that the maximum degree $\Gk(n,m)$ is at least $\Delta$ a.a.s.\
 
  In view of the above, 
  our approach cannot be extended to $d = O(\log n)$ in part~\eqref{thm:hamil:1}
  of Theorem~\ref{thm:hamil}. Nevertheless, we believe (as it was already stated in~\cite{DFRS})
   that for loose Hamilton cycles it suffices to assume that $d=\Omega(1)$.

\begin{conjecture}
For every $k \ge 3$ there is a constant $d_k$ such that if $d \ge d_k$, then a.a.s. $\Rk(n,d)$
contains a loose Hamilton cycle.
\end{conjecture}

We also believe that the lower bounds on $d$ in parts \eqref{thm:hamil:2} and \eqref{thm:hamil:3}
of Theorem~\ref{thm:hamil} are of optimal order.

\begin{conjecture}
For all integers $k>\ell \ge 2$ if $d \ll n^{\ell-1}$, then a.a.s. $\Rk(n,d)$ is \emph{not}
$\ell$-Hamiltonian.
\end{conjecture}

	\section*{Acknowledgements}
	The authors would like to thank the anonymous referees for careful reading of the manuscript and numerous helpful suggestions.
  \bibliographystyle{abbrv} 
  \bibliography{embedding_hypergraphs}

\end{document}